\documentclass{amsart}
\usepackage{amsfonts}
\usepackage{latexsym}
\usepackage{amssymb}
\usepackage{amsmath}

\newcommand{\R}{\mathbb R}

\newcommand{\Pro}{\mathbb P}

\newtheorem{thm}{Theorem}[section]
\newtheorem{cor}{Corollary}[section]
\newtheorem{lemma}{Lemma}[section]
\newtheorem{df}{Definition}[section]
\theoremstyle{remark}
\newtheorem*{rmk}{Remark}

\begin{document}


\title{About the isotropy constant of random convex sets}

\author{David Alonso-Guti\'{e}rrez*}

\thanks{*Supported by FPI Scholarship from DGA (Spain)}

\email{498220@celes.unizar.es}

\address{Universidad de Zaragoza}

\date{July 2007}

\begin{abstract}
Let $K$ be the symmetric convex hull of $m$ independent random vectors uniformly distributed on the unit sphere of $\R^n$. We prove that, for every $\delta>0$, the isotropy constant of $K$ is bounded by a constant $c(\delta)$ with high probability, provided that $m\geq (1+\delta)n$. This result answers a question raised in \cite{KLKO}.
\end{abstract}

\maketitle

\section{Introduction and notation}
A convex body $K\subset \R^n$ is a compact convex set with non-empty interior. A convex body is said to be in isotropic position if it has volume 1 and satisfies the following two conditions:
\begin{itemize}
\item$\int_Kxdx=0 \textrm{ (center of mass at 0)}$
\item$\int_K\langle x,\theta\rangle^2dx=L_K^2\quad \forall \theta\in S^{n-1}$
\end{itemize}
where $L_K$ is a constant independent of $\theta$, which is called the isotropy constant of $K$. Here $\langle\cdot,\cdot\rangle$ denotes the standard scalar product in $\R^n$.

\bigskip
It is well known that for every convex body $K\in\R^n$ there exists an affine map $T$ such that $TK$ is isotropic. Furthermore, $K$ and $TK$ are both isotropic if and only if $T$ is an orthogonal transformation. In this case the isotropy constant of both $K$ and $TK$ is the same. Hence, we can define the isotropy constant for every convex body. It verifies the following equation:
$$
nL_K^2=\min\left\{\frac{1}{|TK|^{1+\frac{2}{n}}}\int_{TK}|x|^2dx \, |\,T\in GL(n) \right\},
$$
where $|\cdot|$ denotes both the volume of a set in $\R^n$ and the euclidean norm of a vector (see  \cite{MP}).

\bigskip
It is conjectured that there exists an absolute constant $C$ such that for every convex body $K$, $L_K\leq C$. This conjecture has been verified for several classes of convex bodies such as unconditional bodies, zonoids and others. However, the best known general upper bound is $L_K\leq cn^{\frac{1}{4}}$,\cite{KL},  which improves the earlier estimate $L_K\leq cn^{\frac{1}{4}}\log n$ given by Bourgain  (see \cite{B}). In a recent paper \cite{KLKO}, Klartag and Kozma proved that with high probability the conjecture is true for the convex hull of independent gaussian vectors and suggested that with the same techniques it should be possible to obtain an analogous result for the convex hull of independent points uniformly distributed on the sphere (in this case the coordinates are not independent). In this paper we are going to prove this result. To be precise, we are going to prove the following theorem:
\begin{thm}\label{teoremaprinc}
For every $\delta>0$, there exist constants $c(\delta)$, $c_1$ and $c_2$, such that if $m>(1+\delta)n$, $\{P_i\}_{i=1}^m$ are independent random vectors on $S^{n-1}$ and $K= \textrm{conv}\{\pm P_1,\dots,\pm P_m\}$, then
$$
\Pro\{L_{K}\leq c(\delta)\}\geq 1-c_1e^{-c_2n\min\{1,\log\frac{m}{n}\}}.
$$
\end{thm}

\bigskip
Along this paper we will denote by $\sigma$ the uniform probability measure on the sphere $S^{n-1}$ and by $\omega_n$ the volume of the euclidean ball. $\Delta^n$ will be the regular simplex of dimension $n$ and the letters $C, c, c_1, c_2,\dots$ will always denote absolute constants whose value may change from line to line.


\section{Some previous results}
\bigskip

In this section we are going to recall some known facts that we are going to use. We begin with the next
\begin{df}
We write $L_{\psi_2}$ for the space of real valued measurable functions on $S^{n-1}$ such that $\int_{S^{n-1}}e^{\frac{|f|^2}{\lambda^2}}d\sigma<\infty$ for some $\lambda>0$, and we set
$$
\Vert f\Vert_{\psi_2}=\inf\left\{\lambda>0\,:\,\int_{S^{n-1}}e^{\frac{|f|^2}{\lambda^2}}d\sigma<2\right\}.
$$
\end{df}

\bigskip
Bernstein's inequality gives a bound for the probability that the absolute value of the sum of $N$ independent random variables with mean 0 is bigger than $\epsilon N$ for certain values of $\epsilon$. There are several versions of this inequality depending on the space where these random variables belong to. We are going to use the following version, whose proof can be found in  \cite{BLM}.

\bigskip
\begin{thm}(Bernstein's inequality)
Assume $\{g_j\}_{j=1}^N\subset L_{\psi_2}$ such that $\Vert g_j\Vert_{L_{\psi_2}}\leq A$ for all $j$ and some constant $A$. Then, for all $\epsilon>0$,
$$
\Pro\left\{\left|\sum_{j=1}^Ng_j\right|>\epsilon N\right\}\leq 2e^{-\frac{\epsilon^2N}{8A^2}}.
$$
\end{thm}

\bigskip
We will apply this inequality to the independent identically distributed random variables $\sqrt{n}\langle P,\theta\rangle$ where $P$ is a point distributed uniformly in the unit sphere. To see that these random variables are in $L_{\psi_2}$ we need the following lemma:

\bigskip
\begin{lemma}
For all $q\geq 1$ and for all $\theta \in S^{n-1}$
$$
\int_{S^{n-1}}|\langle u,\theta\rangle|^qd\sigma(u) = \frac{2\Gamma\left(\frac{1+q}{2}\right)\Gamma\left(1+\frac{n}{2}\right) }{\sqrt{\pi}n\Gamma\left(\frac{n+q}{2}\right)}.
$$

\begin{proof}
Changing to polar coordinates in the following integral, we have that
\begin{align*}
\int_{\R^n}|\langle x,\theta\rangle|^q\frac{e^{-\frac{|x|^2}{2}}}{\left(\sqrt{2\pi}\right)^n}dx&=\frac{n\omega_n}{\left(\sqrt{2\pi}\right)^n}\int_0^\infty r^{n+q-1}e^{-\frac{r^2}{2}}dr\int_{S^{n-1}}|\langle u,\theta\rangle|^qd\sigma(u)\cr
&=\frac{2^{\frac{q}{2}}n\Gamma\left(\frac{n+q}{2}\right)}{2\Gamma\left(1+\frac{n}{2}\right)}\int_{S^{n-1}}|\langle u,\theta\rangle|^qd\sigma(u).
\end{align*}

\bigskip
On the other hand, since
\begin{align*}
\int_{\R^n}|\langle x,\theta\rangle|^q\frac{e^{-\frac{|x|^2}{2}}}{\left(\sqrt{2\pi}\right)^n}dx&=\int_{-\infty}^\infty|x|^q\frac{e^{-\frac{x^2}{2}}}{\sqrt{2\pi}}dx
=\frac{2}{\sqrt{2\pi}}\int_0^\infty x^qe^{-\frac{x^2}{2}}dx\cr
&=\frac{2^{\frac{q}{2}}\Gamma\left(\frac{1+q}{2}\right)}{\sqrt{\pi}},
\end{align*}
we obtain the result.
\end{proof}
\end{lemma}

\bigskip
\begin{rmk}
By Stirling's formula we have that for all $q\geq 1$ and for all $\theta\in S^{n-1}$:
$$
\left(\int_{S^{n-1}}|\langle u,\theta\rangle|^qd\sigma(u)\right)^{\frac{1}{q}}\sim\sqrt{\frac{q}{q+n}}.
$$
\end{rmk}

\bigskip
\begin{cor}
There exists a constant $A>0$ such that for every $\theta\in S^{n-1}$ the functional $\sqrt{n}\langle P,\theta\rangle$ satisfies $\Vert\sqrt{n}\langle \cdot,\theta\rangle\Vert_{\psi_2}\leq A$.
\end{cor}
\begin{proof}
Let $\theta$ be a unit vector, then
\begin{eqnarray*}
\int_{S^{n-1}}e^{\frac{n|\langle P,\theta\rangle|^2}{A^2}}d\sigma(P)&=&\sum_{q=0}^\infty\int_{S^{n-1}}\frac{n^q|\langle P,\theta\rangle|^{2q}}{A^{2q}q!}\\
&\leq& 1 + \sum_{q=1}^\infty \frac{C^{2q}(2q)^{q}}{A^{2q}q!}\\
&\leq& 1 +C'\sum_{q=1}^\infty\left(\frac{2C^2e}{A^2}\right)^q\frac{1}{\sqrt{2\pi q}}\\
&\leq&1 + C'\left(\frac{1}{1-\frac{2C^2e}{A^2}}-1\right)\leq 2
\end{eqnarray*}
if $A>0$ is chosen large enough.
\end{proof}

\bigskip
\section{Symmetric convex hull of random points in the sphere}
\bigskip

In this section we are going to prove theorem \ref{teoremaprinc}.

\bigskip
As it was said in the introduction, the isotropy constant of every convex body verifies the following equation:
$$
nL_K^2=\min\left\{\frac{1}{|TK|^{1+\frac{2}{n}}}\int_{TK}|x|^2dx \, |\,T\in GL(n) \right\}
$$
so, in particular
$$
nL_K^2\leq \frac{1}{|K|^\frac{2}{n}}\frac{1}{|K|}\int_K|x|^2dx
$$

\bigskip
We will prove our result in two steps: we will give a lower bound for $|K|^\frac{1}{n}$ and an upper bound for $\frac{1}{|K|}\int_K|x|^2dx$ which hold with high probability and this will imply our statement.

\bigskip
\begin{lemma}\label{lema}
For every $\delta>0$, there exists a constant $c(\delta)$ such that if $(1+\delta)n<m<ne^{\frac{n}{2}}$, $\{P_i\}_{i=1}^m$ are independent random vectors on $S^{n-1}$, and $K= \textrm{conv}\{\pm P_1,\dots,\pm P_m\}$, then
$$
|K|^{\frac{1}{n}}\geq c(\delta)\frac{\sqrt{\log{\frac{m}{n}}}}{n}
$$
with probability greater than $1-e^{-n}$.
\end{lemma}

\begin{proof}
First observe that, with probability 1, the facets of $K$ are simplices. Let $\alpha \in (0,1)$. If $\alpha B_2^n \nsubseteq K$ then there exists a facet of $K$, which lies in a hyperplane orthogonal to some vector $\theta\in S^{n-1}$  such that $|\langle P_i,\theta \rangle|\leq\alpha$ for all $i$. Let us denote this facet $\textrm{conv}\{Q_1,\dots,Q_n\}$ with $Q_j\in\{\pm P_1,\dots,\pm P_m\}$ and with $Q_i\neq\pm Q_j$. It follows that
$$
\Pro\{\alpha B_2^n \nsubseteq K\}\leq \left(\begin{array}{c}2m\\n\end{array}\right)\Pro\{P\in S^{n-1}\,:\, |\langle P,\theta\rangle|\leq\alpha\}^{m-n}
$$

\bigskip
If $\frac{c}{\sqrt{n}}\leq \alpha\leq \frac{1}{4}$ then
\begin{eqnarray*}
\Pro\{P\in S^{n-1}\,:\, |\langle P,\theta\rangle|>\alpha\}&=&2\frac{(n-1)\omega_{n-1}}{n\omega_n}\int_{\alpha}^1(1-x^2)^{\frac{n-3}{2}}\\
&\geq&2\frac{(n-1)\omega_{n-1}}{n\omega_n}\int_{\alpha}^{2\alpha}(1-x^2)^{\frac{n-3}{2}}\\
&\geq&2\frac{(n-1)\omega_{n-1 }}{n\omega_n}\alpha(1-4\alpha^2)^\frac{n-3}{2}\\
&\geq &c'(1-4\alpha^2)^\frac{n-3}{2}\\
&=&c'e^{\frac{n-3}{2}\log(1-4\alpha^2)}\\
&\geq& c'e^{-4\alpha^2n}.
\end{eqnarray*}

\bigskip
So we have
\begin{align*}
&\Pro\{\alpha B_2^n \nsubseteq K\}\leq\left(\begin{array}{c}2m\\n\end{array}\right)\left(1-c'e^{-4\alpha^2n}\right)^{m-n}\leq\left(\frac{2e m}{n}\right)^ne^{(m-n)\log(1-c'e^{-4\alpha^2n})}\cr
&\leq\left(\frac{2e m}{n}\right)^ne^{-c'(m-n)e^{-4\alpha^2n}}
\end{align*}

\bigskip
Now, set $x=\frac{m}{n}$. Using the fact that there exits a constant $C$ such that $\sqrt{x}<\frac{c'(x-1)}{\log{x}+2+\log{2}}$ for all $x>C$,one can check that if $Cn\leq m\leq ne^{\frac{n}{2}}$ and $\alpha= \frac{1}{2\sqrt{2}}\sqrt{\frac{\log{\frac{m}{n}}}{n}}$ then
$$
\Pro\left\{\frac{1}{2\sqrt{2}}\sqrt{\frac{\log{\frac{m}{n}}}{n}}B_2^n\nsubseteq K\right\}\leq e^{-n}.
$$

\bigskip
On the other hand, if $(1+\delta_0)n\leq m\leq Cn$, since
\begin{eqnarray*}
\Pro\{P\in S^{n-1}\, :\,|\langle P,\theta\rangle|\leq \epsilon\}&=&\frac{2(n-1)\omega_{n-1}}{n\omega_n}\int_0^\epsilon(1-x^2)^{\frac{n-3}{2}}dx\cr
&\leq&c\sqrt{n}\epsilon,
\end{eqnarray*}
we have that if $c_1=c_1(\delta)$ is chosen small enough
\begin{eqnarray*}
\Pro\left\{c_1\sqrt{\frac{\log{\frac{m}{n}}}{n}}B_2^n\nsubseteq K\right\}&\leq&\left(\frac{2em}{n}\right)^n\left(cc_1\sqrt{\log{\frac{m}{n}}}\right)^{m-n}\cr
&\leq&(2eC)^n(cc_1)^{m-n}(\log{C})^\frac{Cn}{2}\cr
&\leq&\left(2eC(cc_1)^{\delta}(\log{C}^\frac{C}{2})\right)^n\cr
&\leq&e^{-n}.
\end{eqnarray*}

\bigskip
Hence
$$
\Pro\left\{\min\left\{c_1(\delta),\frac{1}{2\sqrt{2}}\right\}\sqrt{\frac{\log{\frac{m}{n}}}{n}}B_2^n\nsubseteq K \right\}\leq e^{-n}
$$
and this completes the proof.
\end{proof}

\bigskip
Now let us give an upper bound for $\frac{1}{|K|}\int_{K}|x|^2dx$. It is stated in the next theorem whose proof follows the ideas in \cite{KLKO}.
\begin{thm}\label{teorema}
There exist absolute constants $C$ and $C_1$ such that if $\{P_i\}_{i=1}^m$ are independent random vectors on $S^{n-1}$, $m>n$, and $K= \textrm{conv}\{\pm P_1,\dots,\pm P_m\}$ then
$$
\Pro\left\{\frac{1}{|K|}\int_{K}|x|^2dx\leq C\frac{\log{\frac{m}{n}}}{n} \right\}\geq 1-2e^{-C_1n\log{\frac{m}{n}}}.
$$
\end{thm}
\bigskip

\begin{proof}
By Bernstein's inequality, if $\{P_i\}_{i=1}^n$ are independent random vectors on the sphere and if $\theta$ is a fixed point on the sphere, then, for all $\epsilon >0$,
$$
\Pro\left\{\left|\sum_{i=1}^n \langle P_i,\theta \rangle\right|>\epsilon n\right\}\leq 2e^{-\frac{\epsilon^2n^2}{8A^2}}.
$$

\bigskip
Now let $\mathcal{N}$ be a $\frac{1}{2}$-net on the sphere such that $\left|\mathcal{N}\right|\leq 5^n$. Then
$$
\Pro\left\{\left|\sum_{i=1}^n \langle P_i,\theta\rangle\right|>\epsilon n \textrm{ for some }\theta\in \mathcal{N}\right\}\leq 2e^{-\frac{\epsilon^2n^2}{8A^2}+n\log{5}},
$$
and hence,
$$
\Pro\left\{\left|\sum_{i=1}^n\langle P_i,\theta\rangle\right|\leq\epsilon n \textrm{ for every } \theta\in \mathcal{N}\right\}\geq 1-2e^{-\frac{\epsilon^2n^2}{8A^2}+n\log{5}}
$$

Every $\theta \in S^{n-1}$ can be written in the form $\theta = \sum_{j=1}^\infty \delta_j\theta_j$, with $\theta_j\in\mathcal{N}$ and $0\leq\delta_j\leq \left(\frac{1}{2}\right)^{j-1}$ so, if for every $\theta \in \mathcal{N}$  it is true that $\left|\sum_{i=1}^n\langle P_i,\theta\rangle\right|\leq\epsilon n$, then for every $\theta \in S^{n-1}$
$$
\left|\langle \sum_{i=1}^nP_i,\theta\rangle\right| = \left|\langle \sum_{i=1}^nP_i,\sum_{j=1}^\infty \delta_j\theta_j\rangle\right|\leq \sum_{j=1}^\infty\delta_j|\langle\sum_{i=1}^nP_i,\theta_j\rangle|\leq 2\epsilon n.
$$

Hence
$$
\Pro\left\{\left|\sum_{i=1}^nP_i\right| <2\epsilon n\right\} = \Pro\left\{\max_{\theta\in S^{n-1}}\langle \sum_{i=1}^nP_i,\theta\rangle <2\epsilon n \right\}\geq 1-2e^{-\frac{\epsilon^2n^2}{8A^2}+n\log{5}}.
$$

\bigskip
Now, since $\sum_{i\neq j}\langle P_i,P_j\rangle = \left|\sum_{i=1}^nP_i\right|^2-\sum_{i=1}^n|P_i|^2 = \left|\sum_{i=1}^nP_i\right|^2-n$, this implies that if $\epsilon>\epsilon_0>\sqrt{32A^2\log{5}}$, then
$$
\Pro\left\{\sum_{i\neq j}\langle P_i,P_j\rangle \leq \epsilon n\right\}\geq 1-2e^{-C\epsilon n}.
$$
Since every facet is with probability one $\mathcal{F}_k=\textrm{conv}\{Q_1^k,\dots Q_n^k\}$ with $Q_i^k\in \{\pm P_1,\dots,\pm P_m\}$ and with $Q_i^k\neq \pm Q_j^k$ and since $P$ and $-P$ have the same distribution, if we put $\mathcal{F}_1,\dots,\mathcal{F}_l$ a complete list of the $(n-1)$-dimensional facets of $K$, then
\begin{align*}
\Pro\left\{\max_{k=1\dots l}\sum_{Q_i^k\neq Q_j^k}\langle Q_i^k,Q_j^k\rangle>\epsilon n\log{\frac{m}{n}} \right\}&\leq \left(\begin{array}{c}2m\\n\end{array}\right)2e^{-C\epsilon n\log{\frac{m}{n}}}\cr&\leq2e^{-C\epsilon n\log{\frac{m}{n}} +n\log\left(\frac{2e m}{n} \right)}
\end{align*}
whenever $\epsilon>\epsilon_0$, so choosing a constant $\epsilon$ big enough we have that
$$
\Pro\left\{\max_{k=1\dots l}\sum_{Q_i^k\neq Q_j^k}\langle Q_i^k,Q_j^k\rangle>C n\log{\frac{m}{n}} \right\}\leq 2e^{-C_1n\log\frac{m}{n}}.
$$

\bigskip
For each facet $\mathcal{F}_k = \textrm{conv}\{Q_1^k,\dots,Q_n^k\}$, let $T$ be the following linear transformation:
$$
T=\left(\begin{array}{c c c}Q_1^k(1)&\dots&Q_n^k(1)\\ \vdots& &\vdots \\ Q_1^k(n)&\dots&Q_n^k(n)\end{array}\right).
$$

Then $\mathcal{F}_k=T(\Delta^{n-1})$ and
$$
\frac{1}{|\mathcal{F}_k|}\int_{\mathcal{F}_k}|x|^2dx = \frac{1}{|\Delta^{n-1}|}\int_{\Delta^{n-1}}|Tx|^2dx.
$$

\bigskip
Since
$$Tx=\left(\begin{array}{c}\sum_{i=1}^nQ_i^k(1)x_i\\ \vdots\\ \sum_{i=1}^nQ_i(n)x_i\end{array}\right)$$ then $$|Tx|^2=\sum_{j=1}^n\left(\sum_{i=1}^nQ_i^k(j)x_i\right)^2=\sum_{j=1}^n\sum_{i_1,i_2=1}^nQ_{i_1}^k(j)Q_{i_2}^k(j)x_{i_1}x_{i_2}$$
so
$$
\frac{1}{|\mathcal{F}_k|}\int_{\mathcal{F}_k}|x|^2dx =\frac{1}{|\Delta^{n-1}|}\sum_{j=1}^{n}\sum_{i_1,i_2=1}^nQ_{i_1}^k(j)Q_{i_2}^k(j)\int_{\Delta^{n-1}}x_{i_1}x_{i_2}dx.
$$

\bigskip
From the identity
$$
\frac{1}{|\Delta^{n-1}|}\int_{\Delta^{n-1}}x_{i_1}x_{i_2}dx= \frac{1+\delta_{i_1i_2}}{n(n+1)}
$$
this quantity equals
$$
\frac{1}{n(n+1)}\sum_{j=1}^n\left(\sum_{i=1}^n2Q_{i}^k(j)^2+\sum_{i_1\neq i_2}Q_{i_1}^k(j)Q_{i_2}^k(j)\right)=\frac{2}{n+1}+\frac{1}{n(n+1)}\sum_{i_1\neq i_2}\langle Q_{i_1}^k,Q_{i_2}^k\rangle
$$
and there exist absolute constants $C$ and $C_1$ such that the maximum of this quantity over all facets is less than $C\frac{\log{\frac{m}{n}}}{n}$ with probability bigger than $1-2e^{-C_1 n \log{\frac{m}{n}}}$ so
$$
\sup_{i=1\dots l}\frac{1}{\left|\mathcal{F}_i\right|}\int_{\mathcal{F}_i}|y|^2dy \leq C\frac{\log{\frac{m}{n}}}{n}
$$
with probability bigger than $1-2e^{-C_1 n \log{\frac{m}{n}}}$, where $C$ and $C_1$ are absolute constants.

\bigskip
But, in the same way as it is proved in \cite{KLKO} we have that
\begin{itemize}
\item{$\frac{1}{|\mathcal{K}|}\int_{K}|x|^2dx = \frac{1}{|K|}\sum_{i=1}^l\frac{d(0,\mathcal{F}_i)}{n+2}\int_{\mathcal{F}_i}|y|^2dy$}
\item{$n|K|=\sum_{i=1}^ld(0,\mathcal{F}_i)|\mathcal{F}_i|$}
\end{itemize}
and hence
$$
\frac{1}{|K|}\int_{\mathcal{K}}|x|^2dx\leq \frac{n}{n+2}\sup_{i=1\dots l}\frac{1}{\left|\mathcal{F}_i\right|}\int_{\mathcal{F}_i}|y|^2dy\leq C\frac{\log\frac{m}{n}}{n}
$$
with probability greater than $1-2e^{-C_1n\log\frac{m}{n}}$.
\end{proof}

\bigskip
Lemma \ref{lema} and theorem \ref{teorema} imply that for every $\delta>0$ there exist absolute constants  $c(\delta)$, $c_1$, $c_2$, $C_1$,  such that if $(1+\delta)n<m\leq ne^{\frac{n}{2}}$ then
$$
\Pro\{L_K\leq c\}\geq 1-2e^{-C_1n\log\frac{m}{n}}-e^{-n}>1-c_1e^{-c_2n\min\{1,\log\frac{m}{n}\}}.
$$
Note that in case $m>ne^{\frac{n}{2}}$ we have that
$$
\Pro\{\frac{1}{4}B_2^n\nsubseteq K\}\leq e^{-n}
$$
so, with probability greater than $1-e^{-n}$,
$$
nL_K^2\leq \frac{1}{|K|^{\frac{2}{n}}}\frac{1}{|K|}\int_K|x|^2dx\leq \frac{1}{|\frac{1}{4}B_2^n|^\frac{2}{n}}\leq cn
$$
and so the proof is complete.

\bigskip

{\centerline {ACKNOWLEDGEMENTS}}

\bigskip
This paper was written while the author was in an early stage researcher position of the research training network ``Phenomena in High Dimensions" (MRTN-CT-2004-511953) in Athens. The author would like to thank professor Apostolos Giannopoulos for several helpful discussions as well as for his hospitality.

\end{document}